\documentclass{amsart}
\usepackage{amsmath,amssymb}

\usepackage{amsfonts}
\usepackage{color}
\usepackage{graphicx}
\usepackage[bookmarksnumbered,colorlinks, linkcolor=blue, citecolor=red, pagebackref, bookmarks, breaklinks]{hyperref}

\newtheorem{thm}{Theorem}[section]

\newtheorem{lema}[thm]{Lemma}

\theoremstyle{definition}

\theoremstyle{remark}

\newtheorem{rem}[thm]{Remark}
\numberwithin{equation}{section}

\begin{document}

\title[$\Gamma$-convergence of energy functionals  without the $\Delta_2$ condition]{$\Gamma$-convergence of energy functionals in fractional Orlicz-Sobolev spaces beyond the $\Delta_2$ condition}
	
\author[I. Ceresa Dussel]{Ignacio Ceresa Dussel}

\author[J. Fern\'andez Bonder]{Julian Fern\'andez Bonder}

\author[A. Salort]{Ariel Salort}

\address[I. Ceresa Dussel and J. Fern\'andez Bonder]{Instituto de C\'alculo (UBA - CONICET) and \hfill\break\indent Departamento  de Matem\'atica, FCEyN, Universidad de Buenos Aires, \hfill\break\indent Pabell\'on I, Ciudad Universitaria (1428), Buenos Aires, Argentina.}

\address[A. Salort]{ Departamento de Matem\'aticas y Ciencia de Datos, Universidad San Pablo-CEU, CEU Universities, Urbanizaci\'on Montepr\'incipe, 28660 Boadilla del Monte, Madrid, Spain. }

\medskip

\email[I. Ceresa Dussel]{{\tt iceresad@dm.uba.ar}}

\email[J. Fern\'andez Bonder]{jfbonder@dm.uba.ar}
\urladdr{http://mate.dm.uba.ar/~jfbonder}

\email[A. Salort]{ariel.salort@ceu.es}
\urladdr{https://sites.google.com/view/amsalort/}




\thanks{This work was partially supported by UBACYT Prog. 2018 20020170100445BA and by ANPCyT PICT 2019-00985. J. Fern\'andez Bonder and A. Salort are members of CONICET}

\begin{abstract}
Given a Young function $A$ and $s\in(0,1)$ we consider the energy functional
$$
\mathcal{J}_s(u)=(1-s)\iint_{\mathbb{R}^n\times \mathbb{R}^n} A\left(\frac{|u(x)-u(y)|}{|x-y|^s}\right)\frac{dxdy}{|x-y|^n}.
$$
Without assuming the $\Delta_2$ condition on $A$ nor on its conjugated function $\bar A$, we prove that
$
\mathcal{J}(u) \leq \liminf_{k\to\infty } \mathcal{J}_{s_k}(u_k),
$
where $\mathcal{J}$ is a limit functional related with the behavior of the fractional Orlicz-Sobolev spaces as $s\to 1^-$. As a consequence, we obtain the $\Gamma$-convergence of the functional $\mathcal{J}_s$ on arbitrary open sets, the limit functional being $\mathcal{J}$ restricted to its modular domain.

We also study a peridynamic energy functional and show that, without the $\Delta_2$ condition, the limiting problem becomes trivial, yielding either $0$ or $\infty$.
\end{abstract}

\maketitle

\section{Introduction}
Nonlocal energy functionals given in terms of fractional parameter $s\in (0,1)$ and having the form
\begin{equation} \label{func.p}
(1-s)\iint_{\Omega\times\Omega} \frac{|u(x)-u(y)|^p}{|x-y|^{n+sp}}\,dxdy
\end{equation}
with $\Omega\subset \mathbb{R}^n$, $n\geq 1$,  a smooth bounded domain and $1\leq p <\infty$ play a significant role in   nonlinear analysis. The behavior of \eqref{func.p} as $s\nearrow 1$ in the $W^{1,p}$ norm (resp. BV for $p=1$) was addressed in the seminal work of Bourgain, Brezis and Mironescu \cite{BBM} (see also \cite{Davila}). In recent years, several extensions and variants of these results have been developed. For instance, see  \cite{BMR, CFB, DD,FBS2,K, M,  Ponce}.

In recent years, there has been an increasing attention to operators that do not follow  power-like laws. A prototype energy functional in this context can be described using a so-called Young function $A\colon \mathbb{R}^+\to \mathbb{R}$ and a fractional parameter $s\in (0,1)$ as follows
\begin{equation} \label{func}
\mathcal{J}_s (u) = (1-s)\iint_{\mathbb{R}^{2n}} A(|D^s u|)\,d\nu_n
\end{equation}
where $d\nu_n = |x-y|^{-n}dxdy$ and $D^s u(x,y)=\frac{u(x)-u(y)}{|x-y|^s}$.
The behavior of \eqref{func} as well as many possible extensions, was addressed recently in several articles, see for instance \cite{ACPS,CMSV, FBS, FBS1, MSV}. In particular, \eqref{func} converges to a limit functional as $s\nearrow 1$. When both $A$ and its conjugated function $\bar A$ satisfy the $\Delta_2$ condition, or which is equivalent, when there exist constants $p^+$ and $p^-$ such that
\begin{equation} \label{cond}
1<p^- \leq \frac{tA'(t)}{A(t)} \leq p^+ <\infty,
\end{equation}
a proof of such convergence can be found in  \cite[Theorem 4.1]{FBS}. For an arbitrary Young function the proof is contained in \cite[Theorem 1.1]{ACPS}. More precisely, these results state that for any function $u\in E^A(\mathbb{R}^n)$ it holds that
\begin{equation} \label{teo.conv}
\lim_{s\nearrow 1}\mathcal{J}_s(u) = \mathcal{J} (u)
\end{equation}
where the limit energy function $\mathcal{J}$ is defined as 
$$
\mathcal{J}(u) = \int_{\mathbb{R}^n}  A_0(|\nabla u|)\,dx
$$
where $A_0$ is the Young function defined as follows
\begin{equation}\label{G_0}
A_0(t)=\int_0^t \int_{\mathbb{S}^{n-1}} A(r |\omega\cdot e_n|)\,d \mathcal{H}^{n-1}(\omega)\frac{dr}{r},
\end{equation}
where $\omega$ denotes any fixed unit vector in $\mathbb{S}^{n-1}$. For the precise definition of the functional spaces see Section \ref{sec.prelim}.

The $\Gamma$-convergence of functionals   is a very useful tool in the study of minimum problems arising in the calculus of variations. This is due to the main features of this type of convergence which implies the convergence of minima (see for instance\cite{DM}). 

Regarding the energy functional \eqref{func}, when  \eqref{cond} is assumed, in \cite[Theorem 6.5]{FBS} it is proved that
\begin{equation} \label{teo.gamma.conv1}
\Gamma-\lim_{s\nearrow 1} \bar{\mathcal{J}}_s =\bar{\mathcal{J}}
\end{equation}
where $\bar{\mathcal{J}},\bar{\mathcal{J}}_s \colon L^A(\Omega) \to \bar{\mathbb{R}}$ are defined by
\begin{align*}
\bar{\mathcal{J}}_s(u) = 
\begin{cases}
\mathcal{J}_s(u) & \text{ if } u\in W^{s,A}_0(\Omega)\\
+\infty & \text{ otherwise},
\end{cases}
\qquad 
\bar{\mathcal{J}}(u) = 
\begin{cases}
\mathcal{J}(u) & \text{ if } u\in W^{1,A_0}_0(\Omega)\\
+\infty & \text{ otherwise}
\end{cases}
\end{align*}
where $\Omega\subseteq \mathbb{R}^n$ is open, $W^{s,A}_0(\Omega)$, $s\in(0,1]$ stand for the fractional Orlicz Sobolev spaces defined in Section \ref{sec.prelim}.

The proof of this result follows by proving the two necessary conditions in the definition of $\Gamma$-convergence, namely:
\begin{itemize}
\item[(i)] (\text{lim inf inequality}) For every sequence $\{u_k\}_{k\in\mathbb{N}}\subset L^A(\Omega)$ such that $u_k\to u$ in $L^A(\Omega)$ and $\{s_k\}_{k\in\mathbb{N}}$ such that $s_k\nearrow 1$,
$$
\bar{\mathcal{J}}(u) \leq \liminf_{k\to\infty} \bar{\mathcal{J}}_{s_k}(u_k).
$$
\item[(ii)] (\text{lim sup inequality}) For every $u\in L^A(\Omega)$ there exists a sequence $\{u_k\}_{k\in\mathbb{N}}\subset L^A(\Omega)$ converging to $u$ such that 
$$
\bar{\mathcal{J}}(u) \geq \limsup_{k\to\infty} \bar{\mathcal{J}}_{s_k}(u_k)
$$
where $\{s_k\}_{k\in\mathbb{N}}$ is such that $s_k\nearrow 1$.
\end{itemize}
The proof of condition (ii)  follows by choosing the constant sequence in \eqref{teo.conv}. However, the challenging part lies in proving the liminf inequality, which cannot be adequately addressed solely by \eqref{teo.conv}. A more refined convergence result is required, one that allows the description of the convergence of sequences of functions. 

Under the  \eqref{cond}, the liminf inequality is addressed in \cite[Theorem 5.1]{FBS}. The proof of that result heavily relies on the doubling condition, making highly non trivial to replicate the same proof without assuming the $\Delta_2$ condition. Hence, to the best of our knowledge, we believe this limitation constitutes the primary obstacle preventing the treatment of the $\Gamma$-convergence of \eqref{func} without assuming the $\Delta_2$ condition on $A$ and its conjugate function $\bar{A}$.

The main aim of this article is to fill this gap. First, we prove that the following liminf inequality holds for the functional  $\mathcal{J}_s$ for an arbitrary Young function $A$ without any growth assumptions. More precisely, in Theorem \ref{teo.liminf} we prove that for any function $u\in E^A(\mathbb{R}^n)$ and any sequence $\{u_k\}_{k\in\mathbb{N}}\subset E^A(\mathbb{R}^n)$ such that $u_k\to u$ in $E^A(\mathbb{R}^n)$, if $\{s_k\}_{k\in\mathbb{N}}$ converges to 1 as $k\to\infty$, then
$$
\mathcal{J}(u) \leq \liminf_{k\to\infty} \mathcal{J}_{s_k}(u_k).
$$
As a consequence, in Theorem \ref{coro1} we obtain our first main result:
\begin{equation} \label{teo.gamma.conv}
\Gamma-\lim_{s\nearrow 1} \bar{\mathcal{J}}_s =\bar{\mathcal{J}}
\end{equation}
where $\bar{\mathcal{J}},\bar{\mathcal{J}}_s \colon E^A(\Omega) \to \bar{\mathbb{R}}$ are defined by
\begin{align} \label{J.lim}
\bar{\mathcal{J}}_s(u) =
\begin{cases}
\mathcal{J}_s(u) & \text{ if } u\in W^s_0 L^A(\Omega)\\
+\infty & \text{ otherwise}
\end{cases}
\qquad
\bar{\mathcal{J}}(u) =
\begin{cases}
\mathcal{J}(u) & \text{ if } u\in W^1_0 L^A(\Omega)\\
+\infty & \text{ otherwise,}
\end{cases}
\end{align}
where $\Omega\subseteq\mathbb{R}^n$ is an arbitrary open set, $s\in(0,1]$, and $W^s_0L^A(\Omega)$ denotes the set of functions in $W^sL^A(\mathbb{R}^n)$ that vanish a.e.\ outside $\Omega$ (see Section \ref{sec.prelim}).

We stress two features of this result. First, when $A$ does not satisfy the $\Delta_2$ condition, the effective domain of $\bar{\mathcal{J}}$ is the \emph{modular domain}
$$
\{u\in W^1_0L^A(\Omega)\cap E^A(\Omega)\colon \mathcal{J}(u)<\infty\},
$$
which is in general strictly larger than $W^1_0E^A(\Omega)$ (see Remark \ref{rem.dominio}). Second, no regularity of $\Omega$ is required, and the constant sequence is always a recovery sequence. This follows from the uniform upper bound of Lemma \ref{lema.cota}, which holds on the whole modular domain.

Another interesting topic explored in this work concerns a variant of the functional \eqref{func}. Given  $\delta > 0$, the peridynamic $A$ energy is defined by
\begin{equation}\label{Functional peridinamico}
\mathcal{J}_{\delta}(u) = \int_{\mathbb{R}^n}\int_{B(x,\delta)} A\bigl(|D^s u|\bigr) d\nu_n.
\end{equation}
These functionals are strongly related to peridynamic models, which arise as an alternative class of nonlocal models in continuum mechanics (see \cite{Silling2000,Silling2007}). While there exist general models, the introduction of the horizon $\delta$ provides a length scale that plays a fundamental role in numerical simulations. Furthermore, several works study the limit as $\delta \to 0^+$, showing that the behavior resembles the case $s \to 1^-$, see (\cite{BahrouniBonderCeresaMiyagaki2024,BellidoMoraCorralPedregal2015,BellidoOrtega2021}).

Returning to the energy \eqref{Functional peridinamico}, Ortega in \cite{ort23}, under certain hypotheses of regular variation at the origin and the $\Delta_2$ condition, determined the limit as $\delta \to 0^+$. Moreover, the author obtained $\Gamma$-convergence results. In this work, we show that without assuming the $\Delta_2$ condition and regular variation at the origin, the limit as $\delta \to 0^+$ becomes immediate, the limit function is the Matuszewska function introduced in \cite{MO60}. Consequently, the behavior is not as expected. In fact, although for the functional \eqref{func} the absence of the $\Delta_2$ condition still allows for satisfactory behavior, in the case of the functional \eqref{Functional peridinamico}, the $\Delta_2$ condition proves to be necessary.

\subsection*{Note on this version} A previous version of this article, accepted for publication in \emph{Differential and Integral Equations}, contained two inaccuracies, which are corrected here.
First, in the proof of the liminf inequality (Theorem 3.5 there), the truncation was performed with a factor $(1+\gamma)$, and the final passage to the limit $\gamma\to0$ cannot be justified without the $\Delta_2$ condition. Here we truncate with a factor $(1-\lambda)$ instead (Lemma \ref{lema.truncar}).
Second, in Corollary 3.6 there, the limit functional was restricted to $W^1_0E^A(\Omega)$, which is not correct when $A\notin\Delta_2$. The correct statement is Theorem \ref{coro1}, whose proof is now given in full.

\section{Preliminaries} \label{sec.prelim}

\subsection{Young functions}
A function $A\colon \mathbb{R}_+\to\mathbb{R}_+$ is called a \emph{Young function} if it can be written as
$$
A(t)=\int_0^t a(\tau)\, d\tau,
$$
where $a\colon \mathbb{R}_+\to\mathbb{R}_+$ is nondecreasing, right-continuous, $a(t)>0$ if $t>0$, $a(0)=0$ and $a(t)\to\infty$ as $t\to\infty$. 

From this definition, it can readily be checked that $A$ is convex, $A(0)=0$ and it is superlinear at $0$ and at $\infty$, i.e.
$$
\lim_{t\to 0^+} \frac{A(t)}{t}=0,\quad \lim_{t\to\infty} \frac{A(t)}{t}=\infty.
$$
Observe that the convexity of the Young function immediately gives that
\begin{equation} \label{convexidad}
A(\tau t) \leq \tau A(t) \quad \text{ for } 0<\tau<1, \qquad 
A(\tau t) \geq \tau A(t) \quad \text{ for } \tau>1.
\end{equation}

Given a Young function $A$, an important associated function is the so-called \emph{complementary function}, denoted by $\bar A$, that is defined as
$$
\bar A(t):=\sup \{\tau t- A(\tau)\colon \tau\geq 0\}.
$$
The complementary function $\bar A$ is also a Young function.

Observe that $\bar A$ is the optimal Young function in the {\em Young-type inequality}
$$
\tau t\le \bar A(\tau) + A(t)
$$
for all $\tau,t\geq 0$. It is also easy to see that $\bar{\bar A} = A$.

Given a Young function $A$, if there exist constants $C_0\geq 2$ and $T_0\geq 0$ such that
$$
A(2t)\leq C_0 A(t), \qquad t\leq T_0,
$$
we say that $A$ verifies the $\Delta_2^0$ condition. Similarly, $A$ verifies the $\Delta_2^\infty$ condition if there exist constants  $C_\infty\geq 2$ and $T_\infty\geq 0$ such that
$$
A(2t)\leq C_\infty A(t), \qquad t\geq T_\infty.
$$
More generally, a Young function $A$ is said to verify the $\Delta_2$ condition if there exists  a constant $C\geq 2$ such that
$$
A(2t)\leq CA(t) \quad \text{ for }t\geq 0.
$$

We recall that, given a Young function $A$ it always holds that
$$
A(t)\leq ta(t) \leq A(2t)
$$
which easily follows from the following expressions
$$
A(2t)=\int_0^{2t} a(\tau)\,d\tau >\int_t^{2t} a(\tau)\,d\tau>ta(t), \qquad A(t)=\int_0^t a(\tau)\,d\tau \leq ta(t).
$$

\subsection{Orlicz and Orlicz-Sobolev spaces}
For an introduction to Orlicz spaces we refer to \cite{KR, Gossez}. 

Fractional order Orlicz-Sobolev spaces, as we will use them here, were introduced in \cite{FBS} and then further analysed by several authors. The results
recalled in this section can be found in \cite{ACPS2, ACPS} and in \cite{FBSp24}.

\subsubsection{Orlicz spaces}
Given a open set  $\Omega\subseteq \mathbb{R}^n$ and a Young function $A$, the Orlicz class is defined as
$$
\mathcal{L}^A(\Omega) :=\left\{u\in L^1_\text{loc}(\Omega)\colon \int_\Omega A(|u|)\, dx<\infty\right\}.
$$ 
Then the Orlicz space $L^A(\Omega)$ is defined as the linear hull of $\mathcal{L}^A(\Omega)$ and is characterized as
$$
L^A(\Omega) = \left\{u\in L^1_\text{loc}(\Omega)\colon \text{ there exists } k>0 \text{ s.t. }  \int_\Omega A\left(\frac{|u|}{k}\right)\, dx<\infty\right\}.$$
In general the Orlicz class is strictly smaller than the Orlicz space, and $\mathcal{L}^A(\Omega) = L^A(\Omega)$ if $A$ satisfies the $\Delta_2$ condition.

The space $L^A(\Omega)$ is a Banach space when it is endowed, for instance, with the {\em Luxemburg norm}, i.e.
$$
\|u\|_{L^A(\Omega)} = \|u\|_A :=\inf\left\{k>0\colon  \int_\Omega A\left(\frac{|u|}{k}\right)\, dx\le 1\right\}.
$$

This space $L^A(\Omega)$ turns out to be separable if and only if $A$ satisfies the $\Delta_2$ condition.

An important subspace of $L^A(\Omega)$ is $E^A(\Omega)$ that it is defined as the closure of the functions in $L^A(\Omega)$ that are bounded in $\Omega$ (and  with  compact support if $\Omega=\mathbb{R}^n$).  This space is characterized as
$$
E^A(\Omega) = \left\{u\in L^1_\text{loc}(\Omega)\colon  \int_\Omega A\left(\frac{|u|}{k}\right)\, dx<\infty \text{ for all } k>0\right\}.
$$
This subspace $E^A(\Omega)$ is separable, and we have the inclusions
$$
E^A(\Omega)\subset \mathcal{L}^A(\Omega)\subset L^A(\Omega)
$$
with equalities if $A$ satisfies the $\Delta_2$ condition.

\noindent The following duality relation holds
$$
(E^A(\Omega))^* = L^{\bar A}(\Omega),
$$
where the equality is understood  via the standard duality pairing. Observe that this automatically implies that $L^A(\Omega)$ is reflexive if and only if $A, \bar A$ satisfy the $\Delta_2$ condition.

\subsubsection{Orlicz-Sobolev spaces}
Once the Orlicz spaces are defined, one can easily define the Orlicz-Sobolev spaces in the usual way, namely
$$
W^1L^A(\Omega) = \{u\in L^A(\Omega)\colon \partial_{x_i}u\in L^A(\Omega), i=1,\dots,n\}
$$
and
$$
W^1E^A(\Omega) = \{u\in E^A(\Omega)\colon \partial_{x_i}u\in E^A(\Omega), i=1,\dots,n\}.
$$
These two spaces coincide if and only if $A$ satisfies the $\Delta_2$ condition and in this case we will denote the Orlicz-Sobolev space as
$$
W^{1, A}(\Omega) = W^1L^A(\Omega) = W^1E^A(\Omega).
$$
These spaces are endowed with the natural norm
$$
\|u\|_{W^1L^A(\Omega)} = \|u\|_{1, A} = \|u\|_A + \|\nabla u\|_A.
$$
With this norm, $W^1L^A(\Omega)$ is a Banach space and $W^1E^A(\Omega)$ is a closed subspace.

In order to work with Dirichlet boundary conditions it is necessary to define the Orlicz-Sobolev functions that vanish outside $\Omega$. Since we will not assume any regularity of $\Omega$, we identify functions defined in $\Omega$ with their extensions by zero, and we define
$$
W^1_0L^A(\Omega) := \{u\in W^1L^A(\mathbb{R}^n)\colon u=0 \text{ a.e. in } \mathbb{R}^n\setminus\Omega\}.
$$
We also denote by $W^1_0 E^A(\Omega)$ the closure of $C^\infty_c(\Omega)$ in norm topology. For irregular $\Omega$, the space $W^1_0L^A(\Omega)$ may be larger than the space obtained as the closure of $C^\infty_c(\Omega)$ in the topology $\sigma(W^1L^A, W^{-1}E^{\bar A})$, which is the usual choice for regular domains (see \cite{Gossez, GM02}).

When $A\in \Delta_2$ we have $W^1L^A(\mathbb{R}^n)=W^1E^A(\mathbb{R}^n)$, and we denote $W^{1,A}_0(\Omega)=W^1_0L^A(\Omega)$.

\subsubsection{Fractional Orlicz-Sobolev spaces}
The construction of the fractional order Orlicz-Sobolev spaces is almost the same as the Orlicz-Sobolev spaces. We will make only a sketch and the reader can consult \cite{FBSp24}.

Given a fractional parameter $s\in (0,1)$, we define the H\"older quotient of a function $u\in L^A (\Omega)$ as
$$
D^su(x,y) = \frac{u(x)-u(y)}{|x-y|^s}.
$$
Then, the fractional Orlicz-Sobolev space of order $s$ is defined as
$$
W^sL^A(\mathbb{R}^n) := \{u\in L^A(\mathbb{R}^n)\colon D^su\in L^A(\mathbb{R}^{2n}, d\nu_n)\},
$$
where $d\nu_n = |x-y|^{-n}dxdy$ and
$$
W^sE^A(\mathbb{R}^n) := \{u\in E^A(\mathbb{R}^n)\colon D^su\in E^A(\mathbb{R}^{2n}, d\nu_n)\}.
$$
Since we are now working in the whole space, in order for these spaces to coincide we require that $A\in \Delta_2$. In this case, we denote
$$
W^{s, A}(\mathbb{R}^n) =W^sL^A(\mathbb{R}^n)= W^sE^A(\mathbb{R}^n).
$$
As in the previous case, we have that $W^s L^A(\mathbb{R}^n)$ is reflexive if and only if $A, \bar A\in \Delta_2$.

In these spaces the norm considered is
$$
\|u\|_{W^sL^A(\mathbb{R}^n)} = \|u\|_{s, A} = \|u\|_A + \|D^su\|_{A, d\nu_n}.
$$
Again, with this norm, $W^sL^A(\mathbb{R}^n)$ is a Banach space and $W^sE^A(\mathbb{R}^n)$ is a closed subspace.

As in the case $s=1$, for $s\in(0,1)$ we define
$$
W^s_0L^A(\Omega):=\{u\in W^sL^A(\mathbb{R}^n)\colon u=0 \text{ a.e. in }\mathbb{R}^n\setminus\Omega\},
$$
and we denote by $W^s_0E^A(\Omega)$ the closure of $C^\infty_c(\Omega)$ in $W^sL^A(\mathbb{R}^n)$.

%
%
%

\subsection{Regularized functions}

As usual, we denote by $\rho\in C^\infty_c(\mathbb{R}^n)$ the standard mollifier where $\mathop{\text{\normalfont supp}}(\rho)=B_1(0)$ and $\rho_\varepsilon(x)=\varepsilon^{-n}\rho(\tfrac{x}{\varepsilon})$ is the approximation of the identity. It follows that $\{\rho_\varepsilon\}_{\varepsilon>0}$ is a family of positive functions satisfying 
\begin{equation} \label{regu}
\rho_\varepsilon\in C_c^\infty(\mathbb{R}^n), \quad \mathop{\text{\normalfont supp}}(\rho_\varepsilon)=B_\varepsilon(0), \quad \int_{\mathbb{R}^n} \rho_\varepsilon\,dx=1.
\end{equation}
Given $u\in L^A(\mathbb{R}^n)$ we define the regularized functions $u_\varepsilon\in L^A(\mathbb{R}^n)\cap C^\infty(\mathbb{R}^n)$ as
\begin{equation} \label{regularizada}
u_\varepsilon(x)=u*\rho_\varepsilon(x).
\end{equation}
The following estimate will be of use for our purposes.

\begin{lema} \label{lema.desig}
Let $u\in C^2(\mathbb{R}^n)$ and $\rho\in C^\infty_c(\mathbb{R}^n)$. Then
$$
\|u* \rho\|_{C^2} \leq C_\rho \|u\|_A
$$
where $C$ depends only on $A$, $\rho$ and their derivatives up to order 2.
\end{lema}
\begin{proof}
Let $u\in C^2_c(\mathbb{R}^n)$ and $\rho\in C^\infty_c(\mathbb{R}^n)$. Using H\"older's inequality (see \cite[Theorem 9.3]{KR}) we get that
\begin{align*}
|\partial^\alpha(u*\rho)(x)|=|u* \partial^\alpha \rho(x)| &\leq \int_{\mathbb{R}^n} |u(y)\partial^\alpha \rho(x-y)|\,dy\\
&\leq \|u\|_A \|\partial^\alpha \rho(x-\cdot)\|_{\bar A}\leq C_{\rho,\alpha,A}\|u\|_A 
\end{align*}
since the Orlicz norm is invariant under translations,
where $\partial^\alpha \rho$ denotes a partial derivative of $\rho$ of order $\alpha=(\alpha_1,\ldots, \alpha_n)\in \mathbb{N}_0^n$ with $|\alpha|=\alpha_1+\cdots+\alpha_n \leq 2$. 

This gives the lemma.
\end{proof}

\subsection{The limit Young function}
Recall the definition \eqref{G_0} of $A_0$. The change of variables $r\mapsto tr$ gives the equivalent expression
\begin{equation}\label{G_0.bis}
A_0(t)=\int_0^1 \int_{\mathbb{S}^{n-1}} A(tr |\omega\cdot e_n|)\,d \mathcal{H}^{n-1}(\omega)\frac{dr}{r}.
\end{equation}
In what follows, $n\omega_n=\mathcal{H}^{n-1}(\mathbb{S}^{n-1})$.

\begin{lema}\label{lema.A0}
$A_0$ is a Young function, $\xi\mapsto A_0(|\xi|)$ is convex in $\mathbb{R}^n$, and
$$
c_n\,A(t/4)\le A_0(t)\le n\omega_n\,A(t)\qquad\text{for all } t\ge 0,
$$
where $c_n=\log 2\cdot\mathcal{H}^{n-1}(\{\omega\in \mathbb{S}^{n-1}\colon |\omega\cdot e_n|\ge \tfrac12\})>0$.
\end{lema}

\begin{proof}
The function $\Psi(\tau)=\int_{\mathbb{S}^{n-1}}A(\tau|\omega\cdot e_n|)\,d\mathcal{H}^{n-1}$ is convex and nondecreasing, and by \eqref{convexidad}, $\Psi(\tau r)\le r\Psi(\tau)$ for $r\in[0,1]$. Hence, by \eqref{G_0.bis}, $A_0(t)=\int_0^1\Psi(tr)\,\frac{dr}{r}$ is finite, convex and nondecreasing, and so $\xi\mapsto A_0(|\xi|)$ is convex. Since $A(tr|\omega\cdot e_n|)\le A(t)$, the upper bound follows. For the lower bound, restrict the integration to $r\in[\frac12,1]$ and $|\omega\cdot e_n|\ge\frac12$, where $A(tr|\omega\cdot e_n|)\ge A(t/4)$.
\end{proof}

In particular, $\mathcal{J}(u)<\infty$ implies $\nabla u\in L^A(\mathbb{R}^n)$, and $\int A(|\nabla u|)<\infty$ implies $\mathcal{J}(u)<\infty$. Here and in the sequel we use the convention that $\mathcal{J}(u)=+\infty$ if $u$ is not weakly differentiable.

\section{Proofs of the results}

When $A$ satisfies the $\Delta_2$ condition, in \cite[Lemma 2.13]{FBS} it is proved that  the energy functional decreases under regularization of the function. We claim that the same holds without assuming any growth behavior on $A$.

\begin{lema} \label{lema.reg}
Let $A$ be a Young function and let $u\in L^A(\mathbb{R}^n)$ and $\{u_\varepsilon\}_{\varepsilon>0}$ be the family defined in \eqref{regularizada}. Then
$$
\mathcal{J}_s (u_\varepsilon)  \leq  \mathcal{J}_s (u)
$$
for all $\varepsilon>0$ and $0<s<1$.
\end{lema}
\begin{proof}
Since $A$ is convex and using the last property in \eqref{regu} we can use Jensen's inequality to obtain that for each $x\in \mathbb{R}^n$ and $h\in \mathbb{R}^n\setminus\{0\}$ it holds that
\begin{align} \label{re1}
\begin{split}
A\left(\frac{|u_\varepsilon(x+h)-u_\varepsilon(x)|}{|h|^s} \right) &\leq A\left( \int_{\mathbb{R}^n} \frac{|u(x+h-y) - u(x-y)|}{|h|^s} \rho_\varepsilon(y)\,dy \right)\\
&\leq \int_{\mathbb{R}^n} A\left(  \frac{|u(x+h-y)-u(x-y)|}{|h|^s}\right) \rho_\varepsilon(y)\,dy.
\end{split}
\end{align}
Integrating \eqref{re1} over $\mathbb{R}^n$ and using again \eqref{regu}, for each $h\in \mathbb{R}^n\setminus\{0\}$ we get by using Fubini's theorem that
\begin{align*}
\int_{\mathbb{R}^n} A&\left( \frac{|u_\varepsilon(x+h)-u_\varepsilon(x)|}{|h|^s} \right)\frac{dx}{|h|^n}\leq \\ 
&\leq \int_{\mathbb{R}^n} \int_{\mathbb{R}^n} A\left(  \frac{|u(x+h-y)-u(x-y)|}{|h|^s}\right) \rho_\varepsilon(y)\frac{dy\,dx}{|h|^n}\\
&= \int_{\mathbb{R}^n} \left(\int_{\mathbb{R}^n}A\left( \frac{|u(x+h-y) - u(x-y)|}{|h|^s}\right) \frac{dx}{|h|^n}\right)\rho_\varepsilon(y)\,dy\\
&=\int_{\mathbb{R}^n} A\left(  \frac{|u(x+h)-u(x)|}{|h|^s}   \right) \frac{dx}{|h|^n}.
\end{align*}
Finally, integrating the last expression over $\mathbb{R}^n$ and performing the change of variables $y=x+h$ gives the result.
\end{proof}

Under the assumption \eqref{cond}, the convergence of the energy functional $\mathcal{J}_s$ for smooth functions is established in \cite[Lemma 4.3]{FBS}. We claim that the same holds without assuming any growth condition on the Young function $A$, and moreover, the convergence is uniform on bounded sets of $C^2$.

\normalcolor

Since we do not assume the $\Delta_2$ condition, the truncation of a function has to be combined with a multiplication by a factor smaller than one. Otherwise, the resulting estimate involves $\mathcal{J}_s((1+\gamma)u)$ for some $\gamma>0$, which cannot be controlled in terms of $\mathcal{J}_s(u)$.

\begin{lema}\label{lema.truncar}
Let $u\in E^{A}(\mathbb{R}^n)$ and let $\eta\in C^\infty_c(\mathbb{R}^n)$ be such that $\eta=1$ in $B_1(0)$, $\eta=0$ in $B_2(0)^c$, $0\le \eta\le 1$ and $\|\nabla \eta\|_\infty\leq 2$. For any $\delta\in (0,1]$, let $\eta_\delta(x) = \eta(\delta x)$ and define $u_\delta(x) = u(x)\eta_\delta(x)$.

Then, for any $s\in (0,1)$ and $\lambda\in(0,1)$,
\begin{align*}
\mathcal J_s((1-\lambda)u_\delta)\le{}& (1-\lambda)\mathcal J_s(u)\\
&+ \lambda n\omega_n\left(\delta\int_{\mathbb{R}^n} A\left(\frac{2|u|}{\lambda}\right) dx+\frac{1-s}{s}\int_{\mathbb{R}^n} A\left(\frac{|u|}{\lambda}\right) dx\right).
\end{align*}
\end{lema}

\begin{proof}
We write
$$
(1-\lambda)(u_\delta(x)-u_\delta(y))=(1-\lambda)\,\eta_\delta(x)(u(x)-u(y))+\lambda\,\frac{1-\lambda}{\lambda}\,u(y)(\eta_\delta(x)-\eta_\delta(y)).
$$
By the convexity and monotonicity of $A$, and since $0\le\eta_\delta\le 1$, we get
$$
A\left(|D^s((1-\lambda)u_\delta)|\right)\le(1-\lambda)A(|D^su|)+\lambda A\left(\frac{|u(y)|\,|D^s\eta_\delta(x,y)|}{\lambda}\right).
$$
Let $h=x-y$. Since $|\eta_\delta(x)-\eta_\delta(y)|\le\min\{2\delta|h|,1\}$, we have $|D^s\eta_\delta(x,y)|\le 2\delta|h|^{1-s}$ if $|h|<1$, and $|D^s\eta_\delta(x,y)|\le |h|^{-s}$ if $|h|\ge1$. In both cases the factor in front of $|u(y)|$ is at most one (recall that $\delta\le1$). Hence, by \eqref{convexidad},
$$
A\left(\frac{|u(y)||D^s\eta_\delta(x,y)|}{\lambda}\right)\le
\begin{cases}
\delta|h|^{1-s}A\left(\frac{2|u(y)|}{\lambda}\right) & \text{if }|h|<1,\\[1mm]
|h|^{-s}A\left(\frac{|u(y)|}{\lambda}\right) & \text{if }|h|\ge1.
\end{cases}
$$
The lemma follows by multiplying by $(1-s)$, integrating against $dy\,\frac{dh}{|h|^n}$, and using that
$$
(1-s)\int_{|h|<1}|h|^{1-s-n}\,dh=n\omega_n,\qquad (1-s)\int_{|h|\ge1}|h|^{-s-n}\,dh=\frac{1-s}{s}\,n\omega_n.
$$
\end{proof}

\begin{lema} \label{lema.converg.unif}
Let $A$ be a Young function and let $\mathcal{B}\subset C_c^2(\mathbb{R}^n)$ be such that there is $C>0$ for which $\|u\|_{C^2}\leq C$ for all $u\in \mathcal{B}$ and that there exists $R>0$ such that $\mathop{\text{\normalfont supp}} u\subset B_R$ for all $u\in \mathcal B$. Then it holds
$$
\mathcal{J}_s  \to \mathcal{J}
$$
uniformly in $\mathcal{B}$ as $s\nearrow 1$.
\end{lema}

\begin{proof}
Let   $A$ be an arbitrary Young function (without any assumption on its growth behavior). 
In \cite[Lemma 3.4]{ACPS} it is proved that   for any $u\in C_c^2(\mathbb{R}^n)$ 
$$
\mathcal{J}_s(u) \leq \mathcal{J}(u) + \frac{1-s}{s}  n\omega_n\int_{\mathbb{R}^n} A(2|u|)\,dx.
$$
On the other hand, for any $x\in \mathbb{R}^n$ fixed, without any assumption on $A$, following \cite[Lemma 4.3]{FBS}  (or also \cite[Theorem 1.1]{ACPS}) it follows that
\begin{equation} \label{conv.puntual}
\lim_{s\nearrow 1} \mathcal{F}_s(x) = A_0(|\nabla u(x)|)
\end{equation}
where 
$$
\mathcal{F}_s(x) := (1-s)\int_{\mathbb{R}^n}A(|D^s u|)\frac{dy}{|x-y|^n}.
$$
In light of \eqref{conv.puntual}, in order to prove that
\begin{equation} \label{conve.J}
\lim_{s\nearrow 1 }\mathcal{J}_s(u)=\mathcal{J}(u)
\end{equation}
it only remains to show the existence of an integrable majorant for $\mathcal{F}_s(x)$. Moreover, an inspection of the proof of \eqref{conv.puntual} reveals that
\begin{equation} \label{cota.error}
\left| A(|D^s u|) - A\left(\left|\nabla u(x) \cdot \frac{x-y}{|x-y|^s} \right|\right) \right| \leq C |x-y|^{2-s},
\end{equation}
where $C$ only depends on the $C^2$-norm of $u$. This gives that the convergence in \eqref{conve.J} is in fact uniform on bounded subsets of $C^2(\mathbb{R}^n)$.

We can assume without loss of generality that $\mathop{\text{\normalfont supp}}(u)\subset B_R(0)$ with $R>1$. Then, without any assumption on $A$, when $|x|<2R$, we can write
\begin{equation} \label{cota1}
|\mathcal{F}_s(x)| = \left(\int_{B_{1}(x)}  + \int_{\mathbb{R}^n \setminus B_{1}(x)}\right) A(|D^s u|)\frac{dy}{|x-y|^n}:=I_1+ I_2.
\end{equation}

Arguing as in \cite[Theorem 4.1]{FBS}, we obtain that
\begin{align*}
I_1 &\le \int_{|h|<1} |h|^{1-s-n}\int_0^1 A(|\nabla u(x+th)|)\,dtdh\\ 
&\leq n\omega_n A(\|\nabla u\|_\infty)\int_0^1 r^{-s}\,dr = A(\|\nabla u\|_\infty)\frac{n\omega_n}{1-s}
\end{align*}
and
\begin{align*}
I_2 &\leq \int_{|h|\geq 1} |h|^{-s-n}(A(|u(x+h)| + |u(x)|))\,dh\\ 
&\leq n\omega_n \int_{1}^\infty r^{-s-1}A(2\|u\|_\infty)\,dr =\frac{n\omega_n}{s} A(2\|u\|_\infty).
\end{align*}
Therefore, when $|x|<2R$, these expressions lead to
\begin{equation} \label{cota2}
|\mathcal{F}_s(x)|\leq n\omega_nA(\|\nabla u\|_\infty)   + \frac{1-s}{s} n\omega_n A(2\|u\|_\infty).
\end{equation}
When $|x|\geq 2R$ the function $u$ vanishes and 
$$
\mathcal{F}_s(x)=(1-s)\int_{B_R(0)} A\left(\frac{|u(y)|}{|x-y|^s}\right)\frac{dy}{|x-y|^n}.
$$
Since $|x-y|\geq |x|-R\geq \tfrac{1}{2}|x|$, (and since $|x|\geq 2$) we get that
\begin{align} \label{cota3}
\begin{split}
|\mathcal{F}_s(x)|&\leq \frac{2^s}{|x|^n}\int_{B_R(0)} A\left( \frac{2^s |u(y)|}{|x|^s}\right)\,dy \\
&\leq 
\frac{2^s}{|x|^{n+s}}\int_{B_R(0)} A\left( 2^s |u(y)|\right)\,dy\\
&\leq 
\frac{2^s}{|x|^{n+\frac12}}\int_{B_R(0)} A\left( 2^s |u(y)|\right)\,dy<\infty,
\end{split}
\end{align}
since $u\in C^2_c(\mathbb{R}^n)$. Then, from \eqref{cota1}, \eqref{cota2} and \eqref{cota3} we get that
$$
|\mathcal{F}_s(x)| \leq C_1\left( \chi_{B_R(0)}(x) + \frac{1}{|x|^{n+\frac12}} \chi_{B_R(0)^c}(x)\right) \in L^1(\mathbb{R}^n)
$$
with $C_1>0$ depending on $n$ and $u$ but independent of $s$. Then \eqref{conve.J} follows from the Dominated Convergence Theorem.
\end{proof}

The following observation is the  key in order to study the $\Gamma$-convergence of $\mathcal{J}_s$.
\begin{rem}
If $\{u_k\}_{k\in \mathbb{N}} \subset C_c^2(\mathbb{R}^n)$ is a sequence of functions uniformly bounded in the $C^2$-norm with uniformly bounded supports such that $u_k\to u\in C^2(\mathbb{R}^n)$ as $k\to\infty$, and $\{s_k\}_{k\in\mathbb{N}}$ converges to  1 as $k\to \infty$, then the uniform convergence established in Lemma \ref{lema.converg.unif} ensures that
$$
\lim_{k\to\infty} \mathcal{J}_{s_k}(u_k) = \mathcal{J}(u).
$$
\end{rem}

Following carefully the ideas in \cite{Ponce}, we prove the following liminf inequality.

\begin{thm} \label{teo.liminf}
Let $u\in E^A(\mathbb{R}^n)$ and $\{u_k\}_{k\in\mathbb{N}}\subset E^A(\mathbb{R}^n)$ be such that $u_k \to u$ in $E^A(\mathbb{R}^n)$. Then, for any sequence $\{s_k\}_{k\in\mathbb{N}}$ satisfying that $s_k\nearrow 1$ as $k\to \infty$, we have that
$$
\mathcal{J}(u) \leq \liminf_{k\to\infty} \mathcal{J}_{s_k}(u_k).
$$
Moreover, if the right-hand side is finite, then $u\in W^1L^A(\mathbb{R}^n)$.
\end{thm}

\begin{proof}
Let $M=\liminf_{k\to\infty}\mathcal{J}_{s_k}(u_k)$. We may assume that $M<\infty$ and, passing to a subsequence, that $\mathcal{J}_{s_k}(u_k)\to M$.

\medskip
\noindent\emph{Step 1. Modular bounds.} For every $c>0$, the convexity of $A$ gives
$$
\int_{\mathbb{R}^n}A(c|u_k|)\,dx\le\frac12\int_{\mathbb{R}^n}A(2c|u|)\,dx+\frac12\int_{\mathbb{R}^n}A(2c|u_k-u|)\,dx.
$$
The last term goes to zero, since $\|u_k-u\|_A\to0$ and $\int A(|g|)\,dx\le\|g\|_A$ whenever $\|g\|_A\le1$. Since $u\in E^A(\mathbb{R}^n)$, we conclude that
\begin{equation}\label{cota.modular}
K_c:=\limsup_{k\to\infty}\int_{\mathbb{R}^n}A(c|u_k|)\,dx\le\frac12\int_{\mathbb{R}^n}A(2c|u|)\,dx<\infty\qquad\text{for every } c>0.
\end{equation}

\medskip
\noindent\emph{Step 2. Regularization and truncation.} Fix $\varepsilon>0$, $\delta\in(0,1]$ and $\lambda\in(0,1)$. Following the notation of \eqref{regularizada} and of Lemma \ref{lema.truncar}, we consider the regularized and truncated functions
$$
u_{k,\varepsilon,\delta} = (u_k)_\varepsilon\eta_\delta,\qquad u_{\varepsilon,\delta}=u_\varepsilon\eta_\delta.
$$
By Jensen's inequality,
$$
\int_{\mathbb{R}^n} A(c|(u_k)_\varepsilon|)\,dx\le\int_{\mathbb{R}^n} A(c|u_k|)\,dx\qquad\text{for every } c>0.
$$
Hence, by Lemmas \ref{lema.truncar} and \ref{lema.reg}, we have that
\begin{equation} \label{des1}
\begin{split}
\mathcal{J}_{s_k}((1-\lambda)u_{k,\varepsilon,\delta})\le{}&(1-\lambda)\mathcal{J}_{s_k}(u_k)\\
&+\lambda n\omega_n\left(\delta\int_{\mathbb{R}^n} A\left(\frac{2|u_k|}{\lambda}\right)dx+\frac{1-s_k}{s_k}\int_{\mathbb{R}^n} A\left(\frac{|u_k|}{\lambda}\right)dx\right).
\end{split}
\end{equation}

\medskip
\noindent\emph{Step 3. Limit as $k\to\infty$.} By Lemma \ref{lema.desig},
$$
\|u_{k,\varepsilon,\delta}-u_{\varepsilon,\delta}\|_{C^2}\le C_{\varepsilon,\delta}\|(u_k-u)*\rho_\varepsilon\|_{C^2}\le \tilde C_{\varepsilon,\delta}\|u_k-u\|_{L^A(\mathbb{R}^n)}\to 0.
$$
Moreover, $\mathop{\text{\normalfont supp}} u_{k,\varepsilon,\delta}\subset B_{2/\delta}(0)$ for every $k\in\mathbb{N}$. Therefore, by Lemma \ref{lema.converg.unif} and the subsequent remark, $\mathcal{J}_{s_k}((1-\lambda)u_{k,\varepsilon,\delta})\to\mathcal{J}((1-\lambda)u_{\varepsilon,\delta})$. Letting $k\to\infty$ in \eqref{des1} and using \eqref{cota.modular}, we obtain
\begin{equation}\label{des2}
\mathcal{J}((1-\lambda)u_\varepsilon\eta_\delta)\le(1-\lambda)M+\lambda n\omega_n\,\delta\, K_{2/\lambda}.
\end{equation}

\medskip
\noindent\emph{Step 4. Limits as $\delta\to0$ and $\lambda\to0$.} Since $|\nabla\eta_\delta|\le2\delta$, we have $\nabla((1-\lambda)u_\varepsilon\eta_\delta)\to(1-\lambda)\nabla u_\varepsilon$ pointwise as $\delta\to0$. By Fatou's lemma and \eqref{des2},
$$
\mathcal{J}((1-\lambda)u_\varepsilon)\le(1-\lambda)M\le M.
$$
Letting $\lambda\to0$, the monotone convergence theorem gives
\begin{equation}\label{cota.eps}
\mathcal{J}(u_\varepsilon)\le M\qquad\text{for every }\varepsilon>0.
\end{equation}

\medskip
\noindent\emph{Step 5. Limit as $\varepsilon\to0$.} By Lemma \ref{lema.A0} and \eqref{cota.eps}, $\int_{\mathbb{R}^n} A(|\nabla u_\varepsilon|/4)\,dx\le M/c_n$ for every $\varepsilon>0$. Since $A$ is superlinear at infinity, the de la Vall\'ee-Poussin criterion implies that $\{\nabla u_\varepsilon\}_{\varepsilon>0}$ is equi-integrable on bounded sets. By the Dunford--Pettis theorem, there is a sequence $\varepsilon_j\downarrow0$ such that $\nabla u_{\varepsilon_j}\rightharpoonup G$ weakly in $L^1_{\mathrm{loc}}(\mathbb{R}^n)$. Since $u_\varepsilon\to u$ in $L^1_{\mathrm{loc}}(\mathbb{R}^n)$, we get $G=\nabla u$.

Finally, the functionals $v\mapsto\int A(|v|/4)\,dx$ and $v\mapsto\int A_0(|v|)\,dx$ are convex, hence weakly lower semicontinuous in $L^1_{\mathrm{loc}}(\mathbb{R}^n)$. Therefore $\int A(|\nabla u|/4)\,dx\le M/c_n$, so that $u\in W^1L^A(\mathbb{R}^n)$, and
$$
\mathcal{J}(u)\le\liminf_{j\to\infty}\mathcal{J}(u_{\varepsilon_j})\le M.
$$
\end{proof}

The following estimate extends \cite[Lemma 3.4]{ACPS}, which is stated for functions in $C^2_c(\mathbb{R}^n)$, to the whole modular domain. It provides the limsup inequality with the constant sequence as recovery sequence.

\begin{lema}\label{lema.cota}
Let $u\in E^A(\mathbb{R}^n)$ be such that $\mathcal{J}(u)<\infty$. Then, for every $s\in(0,1)$,
$$
\mathcal{J}_s(u)\le \mathcal{J}(u)+\frac{1-s}{s}\,n\omega_n\int_{\mathbb{R}^n}A(2|u|)\,dx .
$$
In particular, $\limsup_{s\nearrow1}\mathcal{J}_s(u)\le\mathcal{J}(u)$.
\end{lema}

\begin{proof}
Write $\mathcal{J}_s(u)=(1-s)\int_{\mathbb{R}^n}\int_{\mathbb{R}^n}A\big(|u(x+h)-u(x)|/|h|^s\big)\,dx\,\frac{dh}{|h|^n}$ and split the integral in $h$.

If $|h|<1$, write $h=r\omega$ with $r\in(0,1)$ and $\omega\in\mathbb{S}^{n-1}$. By Lemma \ref{lema.A0}, $\nabla u\in L^1_{\mathrm{loc}}(\mathbb{R}^n)$, so for every fixed $h$ and a.e.\ $x$,
$$
u(x+h)-u(x)=\int_0^1\nabla u(x+th)\cdot h\,dt .
$$
By Jensen's inequality and the invariance of the Lebesgue measure under translations,
\begin{align*}
\int_{\mathbb{R}^n}A\left(\frac{|u(x+h)-u(x)|}{|h|^s}\right)dx&\le\int_{\mathbb{R}^n}\int_0^1A\left(r^{1-s}|\nabla u(x+th)\cdot\omega|\right)dt\,dx\\
&=\int_{\mathbb{R}^n}A\left(r^{1-s}|\nabla u(x)\cdot\omega|\right)dx.
\end{align*}
We now integrate over $\omega$ and $r$ with respect to $(1-s)\,d\mathcal{H}^{n-1}(\omega)\frac{dr}{r}$. The change of variables $\tau=r^{1-s}$ gives $(1-s)\frac{dr}{r}=\frac{d\tau}{\tau}$, with $\tau\in(0,1)$. By rotation invariance, $\int_{\mathbb{S}^{n-1}}A(\tau|\xi\cdot\omega|)\,d\mathcal{H}^{n-1}(\omega)=\int_{\mathbb{S}^{n-1}}A(\tau|\xi||\omega\cdot e_n|)\,d\mathcal{H}^{n-1}(\omega)$. By \eqref{G_0.bis}, the contribution of $\{|h|<1\}$ is therefore bounded by
$$
\int_{\mathbb{R}^n}\int_0^1\int_{\mathbb{S}^{n-1}}A(\tau|\nabla u(x)||\omega\cdot e_n|)\,d\mathcal{H}^{n-1}(\omega)\,\frac{d\tau}{\tau}\,dx=\int_{\mathbb{R}^n}A_0(|\nabla u|)\,dx=\mathcal{J}(u).
$$

If $|h|\ge1$, then $|h|^{-s}\le1$, and the convexity of $A$ together with \eqref{convexidad} give
$$
A\left(\frac{|u(x+h)-u(x)|}{|h|^s}\right)\le\frac{|h|^{-s}}{2}\left(A(2|u(x+h)|)+A(2|u(x)|)\right).
$$
Hence the contribution of $\{|h|\ge1\}$ is bounded by
$$
(1-s)\int_{|h|\ge1}|h|^{-n-s}\,dh\int_{\mathbb{R}^n}A(2|u|)\,dx=\frac{1-s}{s}\,n\omega_n\int_{\mathbb{R}^n}A(2|u|)\,dx,
$$
which is finite since $u\in E^A(\mathbb{R}^n)$.
\end{proof}

We can now state and prove our main result.

\begin{thm} \label{coro1}
Let $\Omega\subseteq\mathbb{R}^n$ be an arbitrary open set, and let $\bar{\mathcal{J}}_s$, $\bar{\mathcal{J}}$ be given by \eqref{J.lim}. Then
$$
\Gamma-\lim_{s\nearrow 1} \bar{\mathcal{J}}_s =\bar{\mathcal{J}}
$$
in $E^A(\Omega)$ endowed with the norm topology. That is, for every sequence $\{s_k\}_{k\in\mathbb{N}}$ such that $s_k\nearrow1$:
\begin{itemize}
\item[(i)] for every sequence $\{u_k\}_{k\in\mathbb{N}}$ such that $u_k\to u$ in $E^A(\Omega)$, we have $\bar{\mathcal{J}}(u)\le\liminf_{k\to\infty}\bar{\mathcal{J}}_{s_k}(u_k)$;
\item[(ii)] for every $u\in E^A(\Omega)$, we have $\limsup_{k\to\infty}\bar{\mathcal{J}}_{s_k}(u)\le\bar{\mathcal{J}}(u)$; that is, the constant sequence is a recovery sequence.
\end{itemize}
\end{thm}

\begin{proof}
(i) Let $M=\liminf_{k}\bar{\mathcal{J}}_{s_k}(u_k)$ and assume that $M<\infty$. Along a subsequence realizing $M$ we have $u_k\in W^{s_k}_0L^A(\Omega)$, so $u_k=0$ a.e.\ in $\mathbb{R}^n\setminus\Omega$. Hence $u_k\to u$ in $E^A(\mathbb{R}^n)$ and $u=0$ a.e.\ in $\mathbb{R}^n\setminus\Omega$. By Theorem \ref{teo.liminf}, $u\in W^1L^A(\mathbb{R}^n)$ and $\mathcal{J}(u)\le M$. Therefore $u\in W^1_0L^A(\Omega)$ and $\bar{\mathcal{J}}(u)=\mathcal{J}(u)\le M$.

(ii) We may assume that $\bar{\mathcal{J}}(u)<\infty$, that is, $u\in W^1_0L^A(\Omega)\cap E^A(\Omega)$ with $\mathcal{J}(u)<\infty$. By Lemma \ref{lema.cota}, $\mathcal{J}_s(u)<\infty$ for every $s\in(0,1)$. Hence $D^su\in L^A(\mathbb{R}^{2n},d\nu_n)$, so that $u\in W^s_0L^A(\Omega)$ and $\bar{\mathcal{J}}_s(u)=\mathcal{J}_s(u)$. Lemma \ref{lema.cota} then gives $\limsup_k\bar{\mathcal{J}}_{s_k}(u)\le\mathcal{J}(u)=\bar{\mathcal{J}}(u)$.
\end{proof}

\begin{rem}\label{rem.dominio}
If $A\notin\Delta_2$, the effective domain of $\bar{\mathcal{J}}$ is strictly larger than $W^1_0E^A(\Omega)$. Indeed, let $\Omega$ be a ball. There exists a bounded function $u$ with compact support in $\Omega$ such that $\int A(|\nabla u|)\,dx<\infty$ but $\nabla u\notin E^A(\mathbb{R}^n)$. In dimension one, take $u'(x)=g(x)-g(-x)$, where $g\in L^A\setminus E^A$ is supported in $(0,\frac14)$ and normalized so that $\int A(|g|)\,dx<\infty$. In higher dimensions, take a product with a cut-off function.

Then $u\notin W^1_0E^A(\Omega)$, while $\bar{\mathcal{J}}(u)=\mathcal{J}(u)\le n\omega_n\int A(|\nabla u|)\,dx<\infty$ by Lemma \ref{lema.A0}. In particular, the $\Gamma$-limit cannot be restricted to $W^1_0E^A(\Omega)$.
\end{rem}

\begin{rem}
The domain of $\bar{\mathcal{J}}_s$ in \eqref{J.lim} has to be defined through the condition ``$u=0$ a.e.\ outside $\Omega$''. If it is defined instead as $W^s_0E^A(\Omega)$, that is, as the norm closure of $C^\infty_c(\Omega)$, then Theorem \ref{coro1}(ii) may fail for irregular $\Omega$, already for $A(t)=t^2$.

For instance, let $\Omega=B_1(0)\setminus\Sigma\subset\mathbb{R}^2$ with $\Sigma=[0,1)\times\{0\}$. For $s>\frac12$, every element of the closure of $C^\infty_c(\Omega)$ in $H^s(\mathbb{R}^2)$ has zero trace on $\Sigma$. A compactness argument then shows that the $\Gamma$-liminf is $+\infty$ at any $u\in C^\infty_c(B_1(0))$ with $u|_\Sigma\not\equiv0$, whereas $\bar{\mathcal{J}}(u)<\infty$.
\end{rem}

\section{Peridynamic Operator}
	A new variant of the functional \eqref{func} aims to analyze the behavior of these functionals under localization. In Ortega \cite{ort23}, given a Young function $A$ and real numbers $\delta>0$ and $0<s<1$ they define 
	\begin{equation}
	\mathcal{J}_{\delta} (u) = \int \displaystyle\int_{\mathbb{R}^n\times B(x,\delta)} A(|D^s u|)\,d\nu_n.
	\end{equation}
	One of the main theorems by Ortega is dedicated to study the behavior as \(\delta \to 0^+\). For this purpose, the author assumes certain properties about the function \(A\). Specifically, \(A\) is required to be a Young function that satisfies the \(\Delta_2\) condition and exhibits some regular variation at the origin, that is
	\begin{equation}\label{regular variation}
		\lim_{t\to 0^+}\frac{A(\lambda t)}{t}=\lambda^p \quad p\in \mathbb{R}_+.
	\end{equation}
 	Under these and other suitable hypotheses, Ortega can state and prove the pointwise  convergence to a certain functional
	$$
	\lim_{\delta\to 0^+}\frac{p(1-s)}{A(\delta^{1-s})}\iint_{\mathbb{R}^n\times B(x,\delta)} A(|D^s u|)\,d\nu_n = K_{n,p}\int_{\mathbb{R}^n}|\nabla u(x)|^p dx.
	$$
	In context of $\Gamma$-convergence the author also provides the lim inf inequality when $\{\delta_k\}_{k\in\mathbb{N}}$ is a sequence such that $\delta_k\searrow 0$. 
	The aim of this section is to show that in absence of the previous hypothesis the problem becomes trivial.
	
	Let us assume that $A$ is a Young function, with no $\Delta_2$ assumption and no regular variation hypothesis. 
	 Given $u \in C_c^2(\mathbb{R}^n)$, inequality \eqref{cota.error} holds, i.e.,
		$$
		\left| A(|D^s u|) - A\left(\left|\nabla u(x) \cdot \frac{x-y}{|x-y|^s} \right|\right) \right| \leq C |x-y|^{2-s}.
		$$
		Since	
		$$\int_{|x-y|\leq \delta}|x-y|^{2-s-n}\,dy=\frac{\delta^{2-s} n \omega_n}{2-s},$$
		it follows that 
		\begin{align*}
			\lim_{\delta\to0^+}&\int_{B(x,\delta)}A\left(\frac{|u(x)-u(y)|}{|x-y|^s}\right)\frac{dy}{|x-y|^n}\\
			&=\lim_{\delta\to0^+}\int_{B(x,\delta)} A\left(\frac{|\nabla u(x)\cdot(x-y)|}{|x-y|^s}\right)\frac{dy}{|x-y|^n}.
		\end{align*}	 
	Now, letting $e_x$ and $e$ denote unit vectors, we can rewrite the integral as
	\begin{align*}
		\int_{B(x,\delta)}&A\left(\left|\nabla u(x)\cdot \frac{x-y}{|x-y|^s}\right|\right)\frac{dy}{|x-y|^n}\\
		&=\int_{B(x,\delta)}A\left(|\nabla u(x)|\left|e_x\cdot \frac{x-y}{|x-y|^s}\right|\right)\frac{dy}{|x-y|^n}\\
		&=\int_{B(0,\delta)}A\left(|\nabla u(x)|\left|e\cdot \frac{h}{|h|^s}\right|\right)\frac{dh}{|h|^n}\\
		&=\int_{B(0,1)}A\left(|\nabla u(x)|\delta^{1-s}\left|e\cdot \frac{w}{|w|^s}\right|\right)\frac{dw}{|w|^n}.
		\end{align*}
	Our goal is to study the following limit
	\begin{align}\label{limit}
		\lim_{\delta\to 0^+}\frac{1}{A(\delta^{1-s})}&\int_{B(0,1)}A\left(|\nabla u(x)|\delta^{1-s}\left|e\cdot \frac{w}{|w|^s}\right|\right)\frac{dw}{|w|^n}.
		\end{align}
In \cite{Maligranda,MO60}, the authors define the Matuszewska function  of a function $A$ at $0$ as
$$
\limsup_{\delta\to 0^+} \frac{A(t\delta)}{A(\delta)}=M_0(t).
$$
When $A\in \Delta_2^0$ it follows that $M_0(t)<\infty$ for all $t>0$. However, without $\Delta_2^0$ condition the Matuszewska function is
$$
M_0(t)=
\begin{cases}
1\text{ if }t=1,
	\\ \infty\text{ if }t>1.\end{cases}
$$	
Moreover, under the hypothesis 
\begin{equation}\label{liminf}
	\liminf_{t \to 0^+}\frac{A(2t)}{A(t)}=\infty,
\end{equation}
it is proved in \cite[Proposition 2.1]{FBS24} that $M_0(t)=0$ for all $t \in [0,1)$.

In order to analyze the limit \eqref{limit},
observe that if $0\leq |\nabla u(x)| < 1$
\begin{align*}
\lim_{\delta\to 0^+}&\int_{B(0,1)}\frac{A\left(|\nabla u(x)|\delta^{1-s}\left|e\cdot \frac{w}{|w|^s}\right|\right)}{A(\delta^{1-s})}\frac{dw}{|w|^n}\\ 
&\leq
		 \limsup_{\delta\to 0^+}\int_{B(0,1)}\frac{A\left(|\nabla u(x)|\delta^{1-s}\left|e\cdot \frac{w}{|w|^s}\right|\right)}{A(\delta^{1-s})}\frac{dw}{|w|^n}
	 \\ &\leq \int_{B(0,1)}\limsup_{\delta\to 0^+}\frac{A\left(|\nabla u(x)|\delta^{1-s}\left|e\cdot \frac{w}{|w|^s}\right|\right)}{A(\delta^{1-s})}\frac{dw}{|w|^n}=0.
\end{align*}
If $|\nabla u(x)|=1$ the limit is related to the measure of the unit ball.
Finally if $|\nabla u(x)|>1$, by \eqref{liminf} and Fatou's lemma we have that
\begin{align*}
	\infty&=\int_{B(0,1)}\liminf_{\delta\to 0^+}\frac{A\left(|\nabla u(x)|\delta^{1-s}\left|e\cdot \frac{w}{|w|^s}\right|\right)}{A(\delta^{1-s})}\frac{dw}{|w|^n}\\
	&\leq 
	\liminf_{\delta\to 0^+} \int_{B(0,1)}\frac{A\left(|\nabla u(x)|\delta^{1-s}\left|e\cdot \frac{w}{|w|^s}\right|\right)}{A(\delta^{1-s})}\frac{dw}{|w|^n}.
\end{align*}
Therefore the limit \eqref{limit} becomes either 0 or $\infty$ without the hypothesis \eqref{liminf}, $\Delta_2$ and \eqref{regular variation}.
We can summarize  this results in the following theorem.
\begin{thm}
	Let A be Young function satisfying \eqref{liminf} and $u\in C^2_c(\mathbb{R}^n)$. Then, for every $x\in \mathbb{R}^n$, we have 
	$$
		\lim_{\delta\to 0^+}\frac{1}{A(\delta^{1-s})}\int_{B(x,\delta)} A(|D^s u|)\frac{dy}{|x-y|^n} = K_n M_0(|\nabla u(x)|),
	$$
	where $K_n$ is a constant related to the measure of the unit ball of $\mathbb{R}^n$. 
\end{thm}

\end{document}